\documentclass[15pt]{amsart}
\usepackage[all,cmtip]{xy}
\usepackage{amsthm}
\usepackage{xypic}
\usepackage{graphicx}
\usepackage{amscd}

\usepackage[colorlinks=false, linkcolor=red]{hyperref}

\newcommand{\M}{\overline{\mathcal{M}}}

\newcommand{\PP}{\mathbb{P}}

\newcommand{\OO}{\mathcal{O}}

\newtheorem{theorem}{Theorem}[section]
\newtheorem{thm}[theorem]{Theorem}

\newtheorem{lemm}[theorem]{Lemma}
\newtheorem{prop}[theorem]{Proposition}

\newtheorem{defi}[theorem]{Definition}

\newtheorem{question}[theorem]{Question}

\newtheorem{remark}[theorem]{Remark}

\numberwithin{equation}{section}

\title []{Triviality and Split of Vector Bundles on Rationally Connected Varieties}

\author {xuanyu pan}
\address{Department of Mathematics, Columbia University, New York, NY 10025}
\email{pan@math.columbia.edu}
\date{\today}
\begin{document}

\maketitle

\begin{abstract}
In this paper, we use the existence of a family of rational curves on a separably connected variety, which satisfies the Lefschetz condition, to give a simple proof of a triviality criterion due to I.Biswas and J.Pedro and P.Dos Santos. We also prove that a vector bundle on a homogenous space is trivial if and only if the restriction of the vector bundle to every Schubert line is trivial. Using this result and the theory of Chern classes of vector bundles, we give a general criterion for a uniform vector bundle on a homogenous space to be splitting. As an application, we prove that a uniform vector bundle on a classical Grassmannian of low rank is splitting.
\end{abstract}

\tableofcontents

\section {Introduction}
The study of uniform vector bundles on homogenous spaces starts from the papers \cite{S}, \cite{VDV}, and \cite{BE}. Recently, the paper \cite{UVF} give a more systematical approach. The proofs usually come with a triviality criterion for vector bundles.

In the preprint \cite{VBR}, it provides a beautiful triviality criterion for a vector bundle on a separably rationally connected variety. In Section $2$, we give a simple proof of this theorem by the existence of a family of rational curves on a separable variety, which satisfies the Lefschetz condition, see \cite[Section 8]{KO} for details. The main theorem is the following,

\begin{thm}
Let $X$ be a smooth projective variety over an algebraically closed field. If $X$ is separably rationally connected, then a vector bundle $V$ on $X$ is trivial if and only if the restriction of $V$ to every rational curve on $X$ is trivial.
\end{thm}

In Section $3$, we prove the irreducibility of some moduli spaces of rational chains on homogenous spaces to give a triviality criterion for vector bundles on homogenous spaces, namely,

\begin{thm}
For a homogenous space $G/P$, a vector bundle $V$ on $G/P$ is trivial if and only if the restriction of $V$ to every Schubert line is trivial.
\end{thm}

Use this triviality criterion and the theory of Chern Classes, we give a splitting criterion for uniform vector bundles of low rank on homogenous spaces.

\begin{thm}
Let $G/P$ be a homogenous space of Picard number one. If its $VRMT$ has good divisibility up to degree $r$, then a uniform vector bundle $V$ on $G/P$ of rank at most $r$ is splitting.
\end{thm}

As an application, we apply this theorem to classical Grassimannians and quardrics in Section $5$.
\begin{thm}
A uniform vector bundle on a classical Grassimannian or quardric $X$ of rank at most $s(X)$ is splitting, for the corresponding $X$ and $s(X)$, see Theorem \ref{example}.
\end{thm}

  \textbf{Acknowledgments.} The author is very grateful for his advisor Prof.~A.~de Jong, without the encouragement of his advisor, the author would not write down this paper. The author also thanks his friends Xue Hang , Dr.~Yang Yanhong, Dr.~Yi Zhu and Dr.~Zhiyu Tian for useful conversations.

\section{Triviality of Vector Bundles and Families of Rational Curves}

We start from a general triviality criterion for a vector bundle on an arbitrary smooth projective variety.

\begin{prop}
Let $X$ be a smooth projective variety over an algebraically closed field. A vector bundle $E$ on $X$ is trivial if and only if the restriction of $E$ to every curve $C\subseteq X$ is trivial.

\end{prop}
\begin{proof}
We prove this proposition by induction on the dimension. For $dim(X)=1$, it is trivial. By the existence of Lefschetz pencil, see \cite{ET}, we always have the following blow-up diagram,
\[\xymatrix{\PP^1 &\widetilde{X}\ar[r]^f\ar[l]_{\pi}& X\\
}\]
where $f$ is a blowing up map along a smooth projective subvariety in $X$ of codimension $2$, the map $\pi$ is a Lefschetz pencil. Therefore, we have $f_*\OO_{\widetilde{X}}=\OO_X$. By induction on dimension, the restriction $f^*V|_{\widetilde{X_s}}$ to a smooth fiber $\widetilde{X_s}$ is trivial where $s\in \PP^1$. Suppose $\widetilde{X}_s$ is a singular fiber with one double singularity. We blow up the singularity and get a rational resolution \[g: \widetilde{X}_s^{Bl}\rightarrow \widetilde{X}_s\] with $g_*(\OO_{\widetilde{X}^{Bl}_s})=\OO_{\widetilde{X}_s}$. By induction, we have that \[g^*(f^*V|_{\widetilde{X}_s})=\OO_{\widetilde{X}^{Bl}_s}^{\oplus n}.\]
where $n$ is the rank of $E$. By the projection formula, it implies that \[f^*V|_{\widetilde{X}_s}=g_*g^*(f^*V|_{\widetilde{X}_s})=\OO_{\widetilde{X_s}}^{\oplus n}\]
Since $\pi$ is flat, there exists a vector bundle $W$ on $\PP^1$ such that $\pi^* W=f^* V$ by base change theorem. We can choose a smooth projective curve $C$ such that
\[\xymatrix{C\ar[r]^h \ar[dr]_p &\widetilde{X} \ar[d]^{\pi}\ar[r]^f & X\\
&\PP^1}
\]
where $p$ is a finite and flat map of degree $N$. Since every vector bundle on $\PP^1$ is splitting, we can assume $W=\oplus_{i=1}^n\OO_{\PP^1}(a_i)$ where $a_i\in \mathbb{Z}$. By the hypothesis, we have \[p^*W\simeq h^*(f^*V)\simeq \OO_C^{\oplus n} \]Suppose there is $a_i$ which is positive, we have a exact sequence
\[0\rightarrow p^*\OO_{\PP^1}(a_i)\rightarrow p^* W \simeq \OO_C^{\oplus n},\]
since the degree of $p^*\OO_{\PP^1}(a_i)$ is $N\cdot a_i$ which is positive, the map between $p^*\OO_{\PP^1}(a_i)$ and $p^*W \simeq \OO_C^{\oplus n}$ has to be $0$. It is absurd. Therefore, all the $a_i$ are non-positive. Since we have
\begin{center}
    the degree of $c_1(p^*W)$ is $N\cdot (\sum_{i=1}^{n} a_i)$ and $c_1(p^*W)=0$,
\end{center}
where is $c_1(p^*W)$ the first Chern class of $p^*W$, it concludes that $\sum_{i=1}^{n} a_i=0$, hence, $a_i=0$, i.e, the vector bundle $W$ is trivial. We have\[V=f_*f^*V=f_*\pi^*W=f_*\OO_{\widetilde{X}}^{\oplus n}=\OO_X^{\oplus n}\]
\end{proof}

For a separable rationally connected variety, we can choose the test curves to be rational curves.

\begin{lemm} \label{lemma0}
Let $X$ be a normal scheme and $B$, $C$ be Artin stacks. Suppose we have the following diagram,
\[\xymatrix{
B\ar@/^1pc/[r]^{\sigma} & C  \ar[l]_{\pi} \ar[r]^f & X\\
}.\]
We have the following hypothesis,
\begin{enumerate}
  \item the morphism $\pi$ is representable and flat with fibers of arithmetic genus $0$,
  \item $\sigma$ is a section of $\pi$ and the image of $f \circ \sigma$ is a point,
  \item the morphism $f$ is proper and its geometric general fibers are nonempty, irreducible and generic reduced.
  \item $C$ is irreducible, reduced and contains an open substack as a scheme.
\end{enumerate}
If the restrictions $f^*V|_{C_s}$ are trivial on fibers $C_s$ of $\pi$, then the vector bundle $V$ is trivial.

\end{lemm}

\begin{proof}
 By the Stein factorization for Artin stacks \cite{STACK} or \cite{OM}, we have \[\xymatrix{C\ar[r]_h\ar@/^1pc/[rr]^f &Z\ar[r]_g &X}\]
where $g:Z\rightarrow X$ is a finite morphism between two integral schemes and $h:C\rightarrow Z$ satisfies $h_*{\OO_C}=\OO_Z$. We claim that $f_*\OO_{C}=\OO_X$. Since $X$ is normal and the hypothesis $(4)$, to prove the claim is sufficient to prove $Z$ and $X$ are birational via $g$. By the hypothesis $(3)$ and $(4)$, the geometrical generic fiber $C_{\overline{K}}$ is irreducible and generic reduced where $K$ is the function field of $X$. We have the following cartesian diagram
\[\xymatrix{C_{\overline{K}}\ar[r]\ar[d]\ar@{}[dr]|-{\Box} & L\otimes \overline{K}\ar@{}[dr]|-{\Box} \ar[d]\ar[r] &\overline{K}\ar[d]\\
C\ar[r]_h & Z\ar[r]_g &X}\]
where $L$ is function field of $Z$ and $C_{\overline{K}}$ is flat over $L\otimes \overline{K}$. If $L$ is a non trivial algebraic extension of $K$, then $C_{\overline{K}}$ can not be irreducible or generic reduced, it is a contradiction with condition $(3)$. Therefore, the map $g$ is birational. The claim is clear. Since $f^*V|_{C_s}$ are trivial by the hypothesis $(1)$( i.e, the fibers $C_s$ is of arithmetic genus $0$ and the representability of $\pi$), we have the following properties by base change theorem, see \cite{STACK},
\begin{enumerate}
  \item $\pi^*\pi_*(f^*V)\rightarrow f^* V$ is an isomorphism,
  \item $\pi_*(f^*V)$ is a vector bundle on $B$.
\end{enumerate}

In other words, there exists a vector bundle $W$ on $B$ such that $\pi^*W=f^*V$.
By the hypothesis $(2)$, we have that $W=\sigma^*\pi^*W=\sigma^*f^*V=E$ where $E$ is a trivial bundle since $f\circ \sigma$ is a point.
It implies that $f^*V=\pi^*W$ is a trivial bundle $\OO_C^{\oplus n}$. By the projection formula, see \cite{STACK}, we have \[V=V\otimes f_*\OO_C=f_*(f^*V)=(f_*\OO_C)^{\oplus n}=\OO_X^{\oplus n}\]
\end{proof}

A variety $X$ is separably rationally connected if it admits a very free rational curves, see \cite{KO1} for the definition. For a separably rationally connected projective varieties, we have families of rational curves which satisfies the Lefschetz condition. We also have very flexibility to choose these families, see \cite[Remark 4]{KO2} and \cite[Section 8]{KO}.

\begin{lemm}\label{lemma1} \cite[Theorem 8.10]{KO}\cite[Theorem 3]{KO2}
For a separably rationally connected variety $X$ over an algebraically closed field $k$, there is a family of very free rational curves
\[\xymatrix{
D\ar@/^1pc/[r]^-{\sigma}  & D\times \PP^1 \ar[l]^-{\pi}    \ar[r]_-f & X
}.\]
such that \begin{enumerate}
            \item $f\circ \sigma$ is a point $x$,
            \item the general fibers of $f$ are reduced and irreducible,

          \end{enumerate}
\end{lemm}

\begin{proof}
 For the irreducibility of the general fibers of $f$, it follows from the existence of a family of rational curves \[\xymatrix{\overline{D}\ar@/^1pc/[r]^{s}&C_{\overline{D}}\ar[l]\ar[r]& X }\] satisfies the Lefschetz condition. In the terminology of \cite[Section 8]{KO}, we pick $D=\overline{D}_{s\mapsto x}$. The reason for the general fibers to be reduced is due to the construction of $\overline{D}$ only involved with gluing free, very free curves and smoothing them, in fact, we can choose such $\overline{D}$ to be an open subvariety of $Hom^{vf}(\PP^1,X)$ which is parameterizing very free curves, therefore, the fibers of $f$ are smooth by \cite[Proposition 4.8]{Debar}. See the proof of \cite[Theorem 3]{KO2} for the details.
\end{proof}

\begin{remark}
 \[\]
\begin{enumerate}
    \item In \cite{KO}, we call such a family $D$ is coming from a family $\overline{D}$ of rational curves which satisfies the Lefschetz condition.
    \item For $char(k)=0$, we know separably rational connectedness is equivalent to rational connectedness, hence, a complete Fano manifold is separably rationally connected. For $char(k)=p$, the rational connectedness can not imply the separably rational connectedness, it is subtle. We even do not know whether any Fano hypersurface is S.R.C.
        \item  For $char(k)=p$, we know is that a general Fano hypersurface is separably rationally connected \cite{Zhu1}, and all the smooth cubic hypersurfaces are separably rationally connected \cite{KO}.

\end{enumerate}

\end{remark}

The following theorem is due to \cite{VBR}, we give a simple proof of this beautiful theorem.
\begin{thm} \label{SRCV}
Let $X$ be a smooth projective variety over an algebraically closed field $k$. If $X$ is separably rationally connected, then a vector bundle $V$ on $X$ is trivial if and only if the restriction of $V$ to every rational curves on $X$ is trivial.
\end{thm}

\begin{proof}
Let $M$ be $(\M_{0,1}(X,e))_{red}$ where $\M_{0,1}(X,e)$ is the Kontsevich moduli stack parameterizing stable maps satisfying\begin{enumerate}
                       \item the domain is of arithmetic genus zero,
                       \item the image is a degree $e$ curve in $X$ and the map has finite automorphisms,
                       \item the domain has one pointed point on the smooth locus.
                     \end{enumerate}Let $\mathcal{C}$ be the universal family $\M_{0,2}(X,e)$ of $\M_{0,1}(X,e)$. By Lemma \ref{lemma1}, we have a morphism $h$ induced the following cartesian diagram
\[\xymatrix{D\times \PP^1 \ar[r] \ar@{}[dr]|-{\Box} \ar[d]^{\pi} & \mathcal{C}_{M}\ar[d]\\
D\ar[r]^-h \ar@/^1pc/[u]^{\sigma}& M \ar@/_1pc/[u]_{\sigma_0}},\]
                     where we can choose $D$ to parameterize the maps with trivial automorphism by \cite[Remark 4]{KO2}. If $Char (k)=0$, then $\M_{0,m}(X,e)$ is a proper Deligne-Mumford stack, for arbitrary characteristic, it is a proper Artin stack with finite inertia group, it contains an open substack as a scheme which is parameterizing stable maps with trivial automorphism, see \cite{FP} and \cite{AK} for the details. Let $B$ be $\overline{Im (h)}$. It is obvious that a vector bundle on a chain of rational curves is trivial if and only if its restriction to every component of the chain is trivial. It is easy to check the family of rational curves
                     \[\xymatrix{ B\ar@/^1pc/[r]^{\sigma'} & \mathcal{C}|_B \ar[l] \ar[r] & X}\]
with a section $\sigma'$ satisfies all the hypotheses of Lemma \ref{lemma0}, hence it implies the theorem immediately.
\end{proof}

\section{Triviality of Vector Bundles on Homogenous Spaces}\label{s3}

From this section on, we assume k be an algebraically closed field of characteristic zero. Let $X=G/P$ be a projective homogeneous space under a semi-simple linear algebraic k-group $G$ with a stabilizer parabolic subgroup $P$. By the Bruhat decomposition, the Picard lattice of X is freely generated by the line bundles associated to the Schubert varieties of codimension one, denoted by $L_1,\ldots L_r$. Dually, the curve classes are generated by Schubert curves, denoted by $C_1,\ldots, C_r$.

\begin{defi}
For numbers $i_1,\ldots ,i_m \in \{1,\ldots r\}$, we define a moduli space of chains of Schubert curves with two pointed points and of the pattern $(i_1,\dots,i_m)$. More precisely,
\[Chn_2(X,i_1,\ldots,i_m)=\M_{0,2}(X,[C_{i_1}])_{ev_2}\times_{X,ev_1} \ldots _{ev_2}\times_{X,ev_1} \M_{0,2}(X,[C_{i_m}])\] $(k=1,2)$, in a natural way. Moreover, we have an obvious evaluation map induced by the first pointed point and the last pointed point of the chains, this moduli space is the following
\[ev_{1,m+1}:Chn_2(X,i_1,\ldots,i_m)\rightarrow X\times X.\]
see the figure below.
\end{defi}

\begin{center}
\includegraphics[scale=0.3]{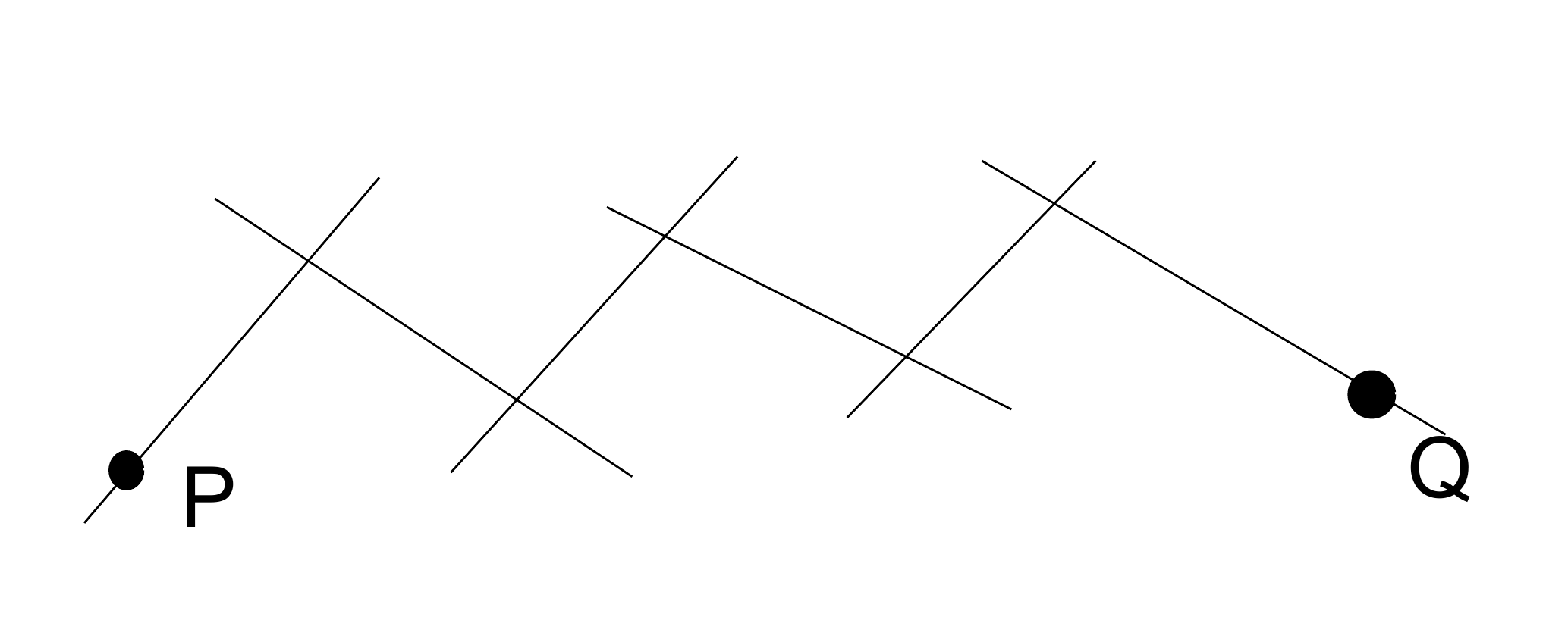}\\
\textit{Figure 0}
\end{center}

\begin{lemm}\label{RC}
The moduli space $Chn_2(X,i_1,\ldots,i_m)$ is a smooth irreducible and rationally connected variety.
\end{lemm}

\begin{proof}
The fibers of $ev:\M_{0,1}(X,[C_i])\rightarrow X$ are smooth irreducible and rationally connected varieties. Indeed, by \cite[Lemma 15.6]{DS} and the fact that any parabolic subgroup is connected, we only need to prove $\M_{0,1}(X,[C_i])$ is rationally connected. If $G$ is simple, then , by \cite[Theorem 4.3]{JL}, the Fano scheme $\M_{0,0}(X,[C_i])$ has two possibilities
\begin{center}
    it is $G/P'$ or the union of two $G$-orbits, an open obit and its boundary $G/P'$
\end{center}
where $P'$ is some parabolic subgroup related to $P$. It implies that $\M_{0,1}(X,[C_i])$ is rational, in particular, rationally connected. In general, see the paper \cite[Theorem 3]{KP}. Since we have the following cartesian diagram
\[\xymatrix{ev_1^{-1}(p)\ar[r]\ar[d]\ar@{}[dr]|-{\Box} &ev^{-1}(p) \ar@{}[dr]|-{\Box} \ar[r]\ar[d] & \{p\}\ar[d]\\
\M_{0,2}(X,[C_i])\ar[r]^{F}\ar@/_2pc/[rr]_{ev_1} & \M_{0,1}(X,[C_i])\ar[r]^-{ev} & X}
\]
where the map $F$ is the forgetful map to forget the second pointed point, therefore, $\M_{0,2}(X,[C_i])$ is the universal bundle bundle over $\M_{0,1}(X,[C_i])$ via $F$, in particular, it is smooth ($\PP^1$ bundle). Hence, the fiber $ev_1^{-1}(p)$ is a $\PP^1$ bundle over $ev^{-1}(p)$ which is a smooth rationally connected projective variety. So $ev_1^{-1}(p)$ is a smooth rationally connected projective variety, by symmetry, so $ev_2^{-1}(p)$ is. By induction on $m$ and the fiber product structure of $Chn_2(X,i_1,\ldots,i_m)$, it is easy to see \[Chn_2(X,i_1,\ldots,i_m)=\M_{0,2}(X,[C_{i_1}])_{ev_2}\times_{X,ev_1} \ldots _{ev_2}\times_{X,ev_1} \M_{0,2}(X,[C_{i_m}])\]
is smooth, irreducible and rationally connected.
\end{proof}

\begin{lemm} \label{Zhu}
There exists $m$ and $i_1,\ldots ,i_m \in \{1,\ldots r\}$ such that the fiber of the evaluation map $ev_{1,m+1}$ is irreducible.
\end{lemm}

\begin{proof}
To prove this lemma, by Lemma \ref{RC} and \cite[Lemma 15.6]{DS}, it is sufficient to prove there exists $m$ and $i_1,\ldots ,i_m \in \{1,\ldots r\}$ such that the evaluation map
 \begin{center}
 $ev_{1,m+1}: Chn_2(X,i_1,\ldots,i_m)\rightarrow X\times X$
\end{center}
is surjective. In other words, it is suffice to prove for any two points $p,q\in X$, we can find a chain of Schubert curves to connect them and the length of the chain is bounded by a number independent of $p$ and $q$. In fact, the proof is more or less similar as the proof in the preprint \cite[Proposition 10.6]{ZHU}. We sketch the proof here. Since the parabolic group $P$ contains a Borel subgroup $B$, by the Bruhat Decomposition, we know the natural map \[G/B\rightarrow G/P\] associates to Schubert curves in $G/B$ Schubert curves in $G/P$. It reduces the problem to $G/B$. When the rank of $G$ is one, the map $ev_{1,m+1}$ is surjective since $G/B=\PP^1$. We consider the natural map
\[q:G/B\rightarrow G/P_{\Delta-\{\alpha\}}\]
where $\Delta$ is the set of simple roots of $G$ and $\alpha\in \Delta$. We know $G/P_{\Delta-\{\alpha\}}$ is of Picard number one. By \cite[Corollary 4.14]{KO1}, we know any two points in $G/P_{\Delta-\{\alpha\}}$ can be connected by a chain of Schubert curves in $G/P_{\Delta-\{\alpha\}}$ of length at most $dim(G/P_{\Delta-\{\alpha\}})$. For any two points $s$, $t\in G/B$, by homogeneity, it is clear that we can lift a chain of Schubert curves in $G/P_{\Delta-\{\alpha\}}$ connecting $q(s)$ and $q(t)$ to a chain of Schubert curves in $G/B$ connecting $s'$ and $t'$, where $s$ and $s'$ (.resp $t$ and $t'$) are in the same fiber of $q$. The fiber $q$ is $P_{\Delta-\{\alpha\}}/B$ which is a homogenous space under a semi-simple linear subgroup of $P_{\Delta-\{\alpha\}}$ of smaller rank. By induction on rank, we can choose chains of Schubert curves to connected $s$ and $s'$ (.resp $t$ and $t'$). These three chains provide a chain of Schubert curves connecting $p$ and $q$.

\end{proof}
We know the invertible sheaf $L = L_1 + \ldots  + L_r$ is ample. Since X is simply connected and homogeneous, by Stein factorization, the invertible sheaf $L$ is in fact very ample. Therefore, using this embedding, Schubert curves are lines in some projective space $\PP^N$, that is why we call it Schubert lines.

\begin{thm} \label{trivial}
For a homogenous space $X=G/P$, under the natural embedding, a vector bundle $V$ on $G/P$ is trivial if and only if the restriction of $V$ to every Schubert line is trivial.
\end{thm}
\begin{proof}
By the previous lemma, we have an evaluation map \[ev_{1,m+1}:Chn_2(X,i_1,\ldots,i_m)\rightarrow X\times X\] where $ev_{1,m+1}=(ev_1,ev_{m+1})$. The evaluation map is surjective with integral fibers. Let $m$ be the smallest number such that $ev_{1,m+1}$ is surjective for some ${i_1,\ldots, i_m}$. The geometrical meanings for $m$ is that it is the shortest length of chain of Schubert lines to connect two general points in $G/P$. Similarly, we can define \[Chn_1(X,i_1,\ldots,i_m)=\M_{0,2}(X,[C_{i_1}])_{ev_2}\times_{X,ev_1} \ldots _{ev_2}\times_{X,ev_1} \M_{0,1}(X,[C_{i_m}])\]
with the unique evaluation map
\[ev:Chn_1(X,i_1,\ldots,i_m)\rightarrow X.\]
We have the following diagram
\[\xymatrix{ D\ar@/^2pc/[rr]^{ev_{m+1}|_D}\ar@{}[ddr]|-{\Box}\ar@{^(->}[r]^{j}\ar@{^(->}[d]\ar[dr] & C\ar[d]_{\pi}\ar[r]^{f} &X\\
 Chn_2(X,i_1,\ldots,i_m)\ar[dr]^F \ar@/_3pc/[drr]_{ev_{1}} &B\ar@/_1pc/[u]_{\sigma} \ar[r]\ar@{}[dr]|-{\Box}\ar@{^(->}[d] & \{pt\}\ar[d]\\
 &Chn_1(X,i_1,\ldots,i_m)\ar[r]^-{ev}& X},\]
 where \begin{enumerate}
         \item $B$ is $ev^{-1}(p)$ and $D$ is $ev_1^{-1}(p)$,
         \item $F$ is the forgetful map to forget the second point on the chains,
         \item $f:C\rightarrow X$ is the universal bundle over $B$ and $\sigma$ is a section of $\pi$ such that $f\circ \sigma=\{p\}$
         \item $j$ is a natural inclusion which corresponds to the last component of $C_b$ over a general point $b\in B$.
       \end{enumerate}

 Since the minimality $m$ and $(4)$ above, we have ,for a general point $q\in X$ ,
 \[f^{-1}(q)=ev_{m+1}|_D^{-1}(q)=ev_{1,m+1}^{-1}(p,q)\]
 which implies $f^{-1}(q)$ is integral.Since it is obvious that a vector bundle of a chain of rational curves is trivial if and only if its restriction to every component of the chain is trivial, we apply Lemma \ref{lemma0} to \[\xymatrix{B \ar@/^1pc/[r]^{\sigma} & C\ar[l]_{\pi}\ar[r]_f &X}\]and complete the proof.
\end{proof}
We provide a simple proof for this theorem base on \cite[Proposition 1.2]{AW}. We sketch a proof here.
\begin{proof}(sketch) We prove this theorem by induction on the rank of $X=G/P$. For rank$=1$, $G/P$ is just $\PP^1$, the proof is trivial. Since $P$ is connected and contains a Borel subgroup $B$, we consider the fibration \[\xymatrix{P/B\ar@{^(->}[r] &G/B \ar[r]^{p} &G/P},\] we have
\begin{center}
 $p_*(\OO_{G/B})=\OO_{G/P}$ and $p$ maps Schubert lines to a point or Schubert lines.
\end{center}
By the projection formula for vector bundles $p^*V$, we reduce the problems to $G/B$. For $G/B$, we apply a similar argument as Lemma \ref{Zhu}. With the same notation as in the proof of Lemma \ref{Zhu}, we consider the following fibration
\[\xymatrix{P_{\Delta-\{\alpha\}}/B\ar@{^(->}[r] &G/B \ar[r]^-{q} &G/P_{\Delta-\{\alpha\}}},\]
the Schubert lines in $P_{\Delta-\{\alpha\}}/B$ is the Schubert lines in $G/B$ via the natural inclusion. Therefore, by induction, we know $V$ is trivial on very fiber of $q$. Since $q$ is smooth with connected proper fibers, by the base change theorem, we have
\begin{center}
$V=q^*W$ for some vector bundle on $W$.
\end{center}
To prove $V$ is trivial is sufficient to prove $W$ is trivial. Since every Schubert line in $G/P_{\Delta-\{\alpha\}}$ can be lifted to a Schubert line in $G/B$, we reduce the problems to $G/P_{\Delta-\{\alpha\}}$ which is of Picard number one. Therefore, it concludes the theorem by \cite[Proposition 1.2]{AW}.

\end{proof}

\section{A Splitting Criterion of Vector Bundles on Homogenous Spaces}

Let $G/P$ be a homogenous space of Picard number one. The ample generator of its Picard group is very ample. Under the embedding induced by this very ample line bundle, we can talk about lines on $G/P$ and  the variety of minimal rational tangents (VRMT for short) on $G/P$. We refer to \cite{HW} for a complete account on the VMRT. In this case, VMRT is just the fiber of the evaluation map
\[ev: \M_{0,1}(G/P,1)\rightarrow G/P.\]
In principle, we can describe the VMRT of a homogenous space by the root data, see \cite[Section 4]{JL} for the details.

\begin{defi}
We call a smooth projective variety $X$ has a good divisibility up to degree $r$ if $x\cdot y=0$, then we have $x=0$ or $y=0$, where $x\in CH^i(X)$, $y\in CH^j(X)$ and $i+j\leq r$.
\end{defi}

\begin{thm}\label{split}
Let $G/P$ be a homogenous space of Picard number one over complex numbers $\mathbb{C}$. If its $VRMT$ has good divisibility up to degree $r$, then a uniform vector bundle $V$ on $G/P$ of rank at most $r$ is splitting.
\end{thm}

\begin{proof}
Since we can twist the vector bundle $V$ with the ample generator of $Pic(G/P)$, by Theorem \ref{trivial}, we can assume the vector bundle $V$ has type $(a_1,a_2,\ldots,a_r)$ when we restrict it to Schubert lines where
\begin{center}
    $a_1=a_2=\ldots=a_k=0$ and $0>a_{k+1}\geq a_{k+2}\geq\ldots \geq a_r$.
\end{center}
We have the following diagram, \[\xymatrix{ & VRMT\ar[r]\ar[d] \ar@{}[dr]|-{\Box} & \{pt\}\ar[d]\\
\M_{0,0}(G/P,1)& C\ar[r]^{ev} \ar[l]_-p & G/P
}\]
where $C$ is the universal bundle of $\M_{0,0}(G/P,1)$, i.e, it is $\M_{0,1}(G/P,1)$. Under our assumption, we have
\[H^0(C_x,ev^*V|_{C_x})=H^0(\PP^1, \left(\oplus^k \OO_{\PP_1}\right)\oplus\left(\oplus_{i= k+1}^r\OO_{\PP^1}(a_i)\right))=\mathbb{C}^k,\]
where $x\in \M_{0,0}(G/P,1)$. Since the dimension does not change when $x$ varies, the push forward $p_*ev^*V$ is a vector bundle of rank $k$ on $\M_{0,0}(G/P,1)$. We claim there exists a short exact sequence of vector bundles as following,
\[0\rightarrow H=p^*p_*ev^*V\rightarrow ev^*V \rightarrow Q\rightarrow 0.\]
In fact, if we restrict to $C_x$, then we get \[H|_{C_x}=\OO_{C_x}^{\oplus k}\subseteq ev^*V|_{C_x}=\OO_{C_x}^{\oplus k}\oplus\left(\oplus^r_{i= k+1}\OO_{C_x}(a_i)\right).\]
In particular, the quotient $Q|_{C_x}$ is a vector bundle $\oplus_{i= k+1}^r\OO_{C_x}(a_i)$, hence, the claim is clear.

The second claim is that there exists vector bundles $V_1$ and $V_2$ on $G/P$ such that
\begin{center}
    $H\cong ev^*V_1$ and $Q\cong ev^*V_2$.
\end{center}
Since the evaluation map is smooth and $VRMT=ev^{-1}(p)$, it is enough to prove that
\begin{center}
    $H|_{VRMT}$ and $Q|_{VRMT}$ are trivial bundles for every $p\in G/P$.
\end{center}
Since the restriction of $ev^*V$ to $VRMT=ev^{-1}(p)$ is trivial for every $p\in G/P$, we get a short exact sequence of vector bundles as following
\[0\rightarrow H|_{VRMT}\rightarrow \underline{\mathbb{C}}^{\oplus r} \rightarrow Q|_{VRMT}\rightarrow 0.\]
Therefore, we have $Q|_{VRMT}$ and $H^{\vee}|_{VRMT}$ are global generated. Therefore, to prove the claim, it is sufficient to prove the first Chern classes are vanishing, i.e, \[c_1(H^{\vee}|_{VRMT})=c_1(Q|_{VRMT})=0.\]
By the Whitney formula of Chern classes and the above short exact sequence, we have $c(H|_{VRMT})\cdot c(Q|_{VRMT})=1$, i.e,
\[(c_k+c_{k-1}+\ldots+1)(\widetilde{c}_{r-k}+\widetilde{c}_{r-k-1}+\ldots+1)=1,\]
where we have
\begin{enumerate}
  \item    $c(H|_{VRMT})=c_k+c_{k-1}+\ldots+1$,
  \item $c(Q|_{VRMT})=\widetilde{c}_{r-k}+\widetilde{c}_{r-k-1}+\ldots+1$.
\end{enumerate}

Since $c_i$ and $\widetilde{c}_j$ are in $CH^*(VRMT)$ and $VRMT$ has good divisibility up to degree $r$, expanding the above equality of Chern classes and comparing both sides, we prove the claim. Therefore, we have the following exact sequence
\[0\rightarrow ev^*V_1\rightarrow ev^*V\rightarrow ev^*V_2 \rightarrow 0\]
Since we know $H|_{C_x}$ is trivial, by Theorem \ref{trivial}, the vector bundle $V_1$ is trivial. Since VRMT is rationally connected, the cohomology group $H^1(VRMT,\OO_{VRMT})$ is trivial. Therefore, apply the functor $ev_*$ to this short exact sequence, we have the following short exact sequence by projection formula and $R^1 ev_*(\OO_C)=0$,
\[0\rightarrow V_1\rightarrow V\rightarrow V_2 \rightarrow 0.\]
By induction on the rank of vector bundles, it implies that $V_2$ is splitting vector bundle. By Theorem \ref{trivial}, we know the vector bundle $V_1$ is trivial. Since $H^1(G/P,L)=0$ for any line bundle on $G/P$ if $dim(G/P)$ is at least $2$, we have that
\[Ext^1(V_2,V_1)=H^1(G/P,(V_2)^{\vee})=0.\] Therefore, we have $V=V_1\oplus V_2$, it completes the proof. For $dim(G/P)=1$, it is $\PP^1$, the theorem is trivial.

\end{proof}

\section{Application to Classical Grassmannians}
In the paper \cite{UVF}, the authors use different methods to prove that uniform vector bundles of low rank on Grassmannians and quadrics are splitting. In this section, we apply Theorem \ref{split} to provide a uniform proof for Grassmannians and quadrics. Applying Theorem \ref{split}, we also prove uniform vector bundles of low rank on other classical Grassmannians to have this similar properties, which can not be deduced from their approach.

\begin{lemm} \label{pp}
Projective space $\PP^m$ has good divisibility up to degree $m$.
\end{lemm}
\begin{proof}
It is obvious since we know $CH^*(\PP^m)=\mathbb{Z}[u]/(u^{m+1})$, see \cite{F}.
\end{proof}

\begin{lemm}\label{div}
Let $X$ and $Y$ be projective varieties. If we have

\begin{enumerate}
  \item $X$ (.resp $Y$) has good divisibility up to degree s (.resp $t$),
  \item $CH^*(X\times Y)=CH^*(X)\otimes CH^*(Y)$,
  \item $CH^*(X)$ and $CH^*(Y)$ are free $\mathbb{Z}$-module,
\end{enumerate}
then $X\times Y$ is has good divisibility up to $min(s,t)$.
\end{lemm}
\begin{proof}
Let $A \in CH^r(X\times Y)$ (.resp $B \in CH^h(X\times Y)$) be
\begin{center}
    $\sum\limits_{k+l=r} a_k \otimes b_l$ (.resp $\sum\limits_{p+q=h} c_p\otimes d_q$)
\end{center}
where \begin{enumerate}
        \item $r+h\leq min(s,t)$ and $l,k, p, q $ are nonnegative,
        \item $a_k\in CH^k(X)$ and $b_l\in CH^l(Y)$, similar for $c_p$ and $d_q$.
      \end{enumerate}
      Suppose $A\cdot B=0$. If $A$ and $B$ are nonzero, then we can assume \begin{itemize}
         \item $a_{k_1}\otimes b_{l_1} \neq 0$ where $k_1=min \{k| a_k\otimes b_l\neq 0\}$,
         \item $c_{p_1}\otimes d_{q_1} \neq 0$ where $p_1=min \{p| c_k\otimes b_l\neq 0\}$.
       \end{itemize}
       It implies that $a_{k_1}\cdot c_{p_1}=0$ or $b_{l_1}\cdot d_{q_1}=0$. Since we have
       \begin{itemize}
         \item $k_1+p_1\leq r+h \leq min(s,t)\leq s,$
         \item $l_1+q_1\leq r+h \leq min(s,t)\leq t,$
       \end{itemize}
       it implies that $a_{k_1}=0$ or $c_{p_1}=0$ or $b_{l_1}=0$ or $d_{q_1}=0$. It is a contradiction.
\end{proof}

\begin{lemm} \label{q1}\cite[Theorem 1.12, Page 32]{REID}
If $m=2n+1$ and $X$ is a $m$-dimensional quadric, then $$1,b,\ldots,b^{n-1},a,ba,\ldots,b^{n+1}a$$ generate $H^{2i}(X,\mathbb{Z})$ and $b^n=2a$ where $a\in H^{2n}(X,\mathbb{Z})$. In particular, the quadric $X$ has good divisibility up to $m$.
\end{lemm}
\begin{lemm}  \label{q2} \cite[Theorem 1.13, Page 33]{REID}
If $m=2n$ and $X$ is a $m$-dimensional quadric, then $$1,c,\ldots,c^{n-1},a \textit{ and } b ,ca,\ldots,c^{n}a$$ generate $H^{2i}(X,\mathbb{Z})$ and $ca=cb$ where $a$ and $b$ $\in H^{2n}(X,\mathbb{Z})$. Moreover,
\begin{enumerate}
  \item $a^2=b^2=1$ and $ab=0$ if $n$ is even,
  \item $a^2=b^2=0$ and $ab=1$ if $n$ is odd,
\end{enumerate}
In particular, the quadric $X$ has good divisibility up to $m-1$.
\end{lemm}
\begin{proof}
Since quadrics are homogenous spaces, by the Bruhat decomposition, we know $H^*(X)=CH^*(X)$, see \cite{F}.
\end{proof}

\begin{lemm} \label{vb}
If $X$ is a nonsingular variety and has good divisibility up to $r$ and $V$ is a vector bundle of rank $n$ on $X$, then $\PP(V)$ has good divisibility up to $min(r,n-1)$.
\end{lemm}

\begin{proof}
By \cite{F}, we have
\[CH^*(\PP(V))=CH^*(X)[u]/(u^n+c_1(V)u+\ldots+c_n(V)).\]
The proof is similar to the proof of Lemma \ref{div}.

\end{proof}
\begin{thm}\label{example}
A uniform vector bundle on $X$ of rank at most $s(X)$ is splitting, where $X$ and $s(X)$ is the following\\
\\

\begin{center}

\begin{tabular}{|c|c|c|}
  \hline
  $X$ & $s(X)$ & V.R.M.T \\
  & & \\ \hline
  $\PP^n$ & n-1 & $\PP^{n-1}$ \\
  & & \\  \hline
  $\mathbb{Q}^n$ & $n-2$ if $n$ is odd & $\mathbb{Q}^{n-2}$  \\
   & $n-3$ if $n$ is even & \\\hline
  $\mathbb{G}(k,n)$ & $min(k,n-k-1)$ & $\PP^k\times \PP^{n-k-1}$ \\
  & & \\ \hline
  $\mathbb{OG}(k,2m-1)$ & $min(k,2m-2k-3)$ & $\PP^k\times \mathbb{Q}^{2m-2k-3}$  \\
  &&for $1\leq k \leq m-2$ \\
  & & \\\hline
    $\mathbb{SG}(k,2m-1)$& $min(m,2m-2k-2)$ &  $\PP_{\PP^m}(\OO(2)\oplus\OO(1)^{2m-2k-2})$  \\
    & & for $1\leq k \leq m-2$ \\
    & & \\\hline
   $\mathbb{SG}(m-1,2m-1)$ & m-1 & $v_2(\PP^{m-1})$ \\\hline
\end{tabular}
\end{center}

\end{thm}

\begin{proof}
By the paper \cite{JL} or the survey  \cite{SURVEY}, we know the $VRMT$ for these examples. Applying Theorem \ref{split} and the lemmas in this section, we get the results. \\
For $\PP^n$, we apply Lemma \ref{pp}.\\
For $\mathbb{Q}^n$, we apply Lemma \ref{q1} and \ref{q2}.\\
For $\mathbb{G}(k,n)$, we apply Lemma \ref{div}.\\
For $\mathbb{OG}(k,2m-1)$, we apply Lemma \ref{div} and \ref{q1}.\\
For $\mathbb{SG}(k,2m-1)$, we apply Lemma \ref{vb}.\\
For $\mathbb{SG}(m-1,2m-1)$, we apply Lemma \ref{pp}.\\

\end{proof}
\begin{question}
Classify the uniform vector bundles on $X$ of rank $s(X)+1$?
\end{question}
See the paper \cite{UVF} for some answers.
\bibliographystyle{alpha}
\bibliography{mybib}
{}

\end{document}